\documentclass[a4paper,10pt, reqno]{amsart}
\usepackage{amssymb,latexsym,amsmath,epsfig,amsthm} 

\theoremstyle{plain}
\newtheorem{theorem}{Theorem}
\newtheorem*{theorem*}{Theorem}

\newtheorem*{corollary*}{Corollary}
\newtheorem{lemma}{Lemma}
\newtheorem*{lemma*}{Lemma}

\newtheorem*{proposition*}{Proposition}

\newtheorem*{conjecture*}{Conjecture}
\theoremstyle{definition}

\newtheorem*{definition*}{Definition}
\theoremstyle{remark}
\newtheorem{remark}{Remark}
\newtheorem*{remark*}{Remark}

\begin{document}
\title[A function defined by automaton with finite memory]{On one nearly everywhere continuous and  nowhere differentiable function, that defined by automaton with finite memory}
\author{Symon Serbenyuk}
\address{Institute of Mathematics \\
 National Academy of Sciences of Ukraine \\
  3~Tereschenkivska St. \\
  Kyiv \\
  01004 \\
  Ukraine}
\email{simon6@ukr.net}



\begin{abstract}

This paper is devoted to the investigation of the following function
$$
f: x=\Delta^{3} _{\alpha_{1}\alpha_{2}...\alpha_{n}...}{\rightarrow}  \Delta^{3} _{\varphi(\alpha_{1})\varphi(\alpha_{2})...\varphi(\alpha_{n})...}=f(x)=y,
$$
where $\varphi(i)=\frac{-3i^{2}+7i}{2}$, $ i \in N^{0} _{2}=\{0,1,2\}$, and $\Delta^{3} _{\alpha_{1}\alpha_{2}...\alpha_{n}...}$ is the ternary  representation of  $x \in [0;1]$. That is values of this function are obtained from the ternary representation
of the argument by the following change of digits: 0 by 0, 1 by 2, and 2 by 1. This function preserves the ternary digit $0$.

Main mapping properties and differential, integral, fractal properties of the function are  studied. Equivalent representations by additionally defined auxiliary functions of this function are proved.

This paper is the paper translated from Ukrainian (the Ukrainian variant \cite{Symon 2012} available at 
  https://www.researchgate.net/publication/292970012). In 2012, the Ukrainian variant \cite{Symon 2012} of this paper was represented by the author in  the International Scientific Conference ``Asymptotic Methods in the Theory of Differential Equations" dedicated to 80th anniversary of M. I. Shkil (the conference paper available at https://www.researchgate.net/publication/301765319). In 2013, the investigations of the  present article were generalized by the author in the paper \cite{S. Serbenyuk functions with complicated local structure 2013} ``One one class of functions with complicated local structure" (https://arxiv.org/pdf/1601.06126.pdf) and in the several conference papers \cite{{S. Serbenyuk abstract 7}, {S. Serbenyuk abstract 8}} ( available at:  https://www.researchgate.net/publication/301765326, \\ https://www.researchgate.net/publication/303052308).

\end{abstract}
\maketitle



\section{Introduction}

Fractal sets use often in the metric theory of numbers and the probability theory of numbers, and the investigations of the mathematical objects with complicated local structure as well. To investigate such sets in space $\mathbb R^2$, it is considered often these sets as graphs of functions that take $\mathbb R^1$ to $\mathbb R^1$.

This article is devoted to  one example of function with complicated local structure such that the argument and values of the function defined in terms of the ternary representation of real numbers.

\section{The Object of Research}

We shall
not consider numbers whose the  ternary representations have the period (2) (without the number $1$).  Consider a certain function  $f$ defined on $[0;1]$ by the following way: 
\begin{equation}
\label{ff1}
x=\Delta^{3} _{\alpha_{1}\alpha_{2}...\alpha_{n}...}\stackrel{f}{\rightarrow} \Delta^{3} _{\varphi(\alpha_{1})\varphi(\alpha_{2})...\varphi(\alpha_{n})...}=f(x)=y,
\end{equation}
where $\varphi(i)=\frac{-3i^{2}+7i}{2}$, $ i \in N^{0} _{2}=\{0,1,2\}$.

In this paper, we consider differential, integral, fractal and other properties of the function~$f$.

We begin with definitions of some  auxiliary functions. Let $i, j, k$ be pairwise distinct digits of the ternary number system.
First let us introduce a function $\varphi_{ij} (\alpha)$ defined on an alphabet of the ternary number system by the following:
\begin{center}
\begin{tabular}{|c|c|c|c|}
\hline
 &$ i$ &$ j $ & $k$\\
\hline
$\varphi_{ij} (\alpha)$ &$0$ & $0$ & $1$\\
\hline
\end{tabular}
\end{center}
That is  $f_{ij}$  is a function given on $[0;1]$ by the following way
$$
x=\Delta^{3} _{\alpha_{1}\alpha_{2}...\alpha_{n}...}\stackrel{f_{ij}}{\rightarrow} \Delta^{3} _{\varphi_{ij} (\alpha_{1})\varphi_{ij}  (\alpha_{2})...\varphi_{ij} (\alpha_{n})...}=f_{ij} (x)=y.
$$

\begin{remark} From the definition of $f_{ij}$ it follows that 
 $f_{01} =f_{10}$, $f_{02} =f_{20}$, $f_{12} =f_{21} $. Since it is true, we shall use only the following denotations:  $f_{01}, f_{02}, f_{12}$.
\end{remark}

\begin{lemma}
\label{lmf1}
The function $f$ can be represented by the following equivalent equalities:
\begin{enumerate}
\item
\begin{equation}
\label{ff2} 
f(x)=2x-3f_{01} (x), ~~\mbox{where}~~\Delta^{3} _{\alpha_{1}\alpha_{2}...\alpha_{n}...}\stackrel{f_{01}}{\rightarrow} \Delta^{3} _{\varphi_{01} (\alpha_{1})\varphi_{01} (\alpha_{2})...\varphi_{01} (\alpha_{n})...},
\end{equation}
$\varphi_{01} (i)=\frac{i^2 - i}{2}$, $i \in N^0 _2$;
\item
\begin{equation}
\label{ff3}
f(x)=\frac{3}{2}-x-3f_{12} (x), ~~\mbox{where}~~\Delta^{3} _{\alpha_{1}\alpha_{2}...\alpha_{n}...}\stackrel{f_{12}}{\rightarrow} \Delta^{3} _{\varphi_{12} (\alpha_{1})\varphi_{12} (\alpha_{2})...\varphi_{12} (\alpha_{n})...},
\end{equation}
$\varphi_{12} (i)=\frac{i^2 - 3i+2}{2}$, $i \in N^0 _2$.

\item
\begin{equation}
\label{ff4}
f(x)=\frac{x}{2}+\frac{3}{2}f_{02} (x), ~~\mbox{where}~~\Delta^{3} _{\alpha_{1}\alpha_{2}...\alpha_{n}...}\stackrel{f_{02}}{\rightarrow} \Delta^{3} _{\varphi_{02} (\alpha_{1})\varphi_{02} (\alpha_{2})...\varphi_{02} (\alpha_{n})...},
\end{equation}
$\varphi_{02} (i)=-i^2 +2i$, $i \in N^0 _2$.
\end{enumerate}
\end{lemma}

\begin{proof} It is easy to see that the ternary representation of each number 
 $x=\Delta^{3} _{\alpha_{1}\alpha_{2}...\alpha_{n}...}$ can be represented by the sum of two numbers 
 $\Delta^{3} _{\beta_{1}\beta_{2}...\beta_{n}...}   $  and  $\Delta^{3} _{\gamma_{1}\gamma_{2}...\gamma_{n}...}$ such that  $\beta_{n} \in N^0 _1$ and $\gamma_{n} \in N^0 _1$ for each $n \in \mathbb N$. Here $\alpha_n=2$ if and only if $\beta_n=\gamma_n=1$.

It is easily shown that   $f(x)=2x$ on the following set  
$$
C[3, \{0,1\}]=\{x: x=\Delta^{3} _{\alpha_{1}\alpha_{2}...\alpha_{n}...},~\alpha_{n} \in \{0,1\}\}
$$
This fact follows from the definition of  $f$.

Since $1=\varphi(2) \ne \varphi(1) + \varphi(1)=4$, we have
$\varphi(2)= \varphi(1) + \varphi(1)-3.$ 
Hence, 
$$
f(x)=f(x_1)+f(x_2)-3\Delta^{3} _{\underbrace{000...00}_{k_1 -1}1\underbrace{000...00}_{k_2-k_1-1}1\ldots \underbrace{000...00}_{k_n-(k_1+\dots+k_{n-1})-1}1\ldots}=
$$
$$
= 2x -3\Delta^{3} _{\underbrace{000...00}_{k_1 -1}1\underbrace{000...00}_{k_2-k_1-1}1...\underbrace{000...00}_{k_n-(k_1+...+k_{n-1})-1}1...} ,~ \mbox{where}
$$
$x=\Delta^{3} _{e_{1}e_{2}...e_{k_1-1}2e_{k_1+1}...e_{k_2-1}2e_{k_2+1}...e_{k_n-1}2e_{k_n+1}...}$, $e_k \in \{0,1\}$.
That is $k_n$ is the position  of the nth digit $2$ in the representation of $x$. 

The last-mentioned representation of the function $f$  and the following representation 
$$
f(x)=2x-3f_{01} (x), 
$$
where
$$
x=\Delta^{3} _{\alpha_{1}\alpha_{2}...\alpha_{n}...}\stackrel{f_{01}}{\rightarrow} \Delta^{3} _{\varphi_{01} (\alpha_{1})\varphi_{01} (\alpha_{2})...\varphi_{01} (\alpha_{n})...}=f_{01}(x)=y,
$$
$$
\varphi_{01}(\alpha_i)=\begin{cases}
0&\text{if $\alpha_i \in\{0,1\}$}\\
1&\text{if $\alpha_i =2$,}
\end{cases}
$$
i.e., $\varphi_{01} (i)=\frac{i^2 - i}{2}$ for $i \in N^0 _2$, are equivalent. 

It is easy to prove that  $f_{01} (x)=x - \Delta^3 _{111...}+f_{12} (x)=x-\frac{1}{2}+f_{12} (x)$, where
$$
x=\Delta^{3} _{\alpha_{1}\alpha_{2}...\alpha_{n}...}\stackrel{f_{12}}{\rightarrow} \Delta^{3} _{\varphi_{12} (\alpha_{1})\varphi_{12} (\alpha_{2})...\varphi_{12} (\alpha_{n})...}=f_{12}(x)=y,
$$
$\varphi_{12} (i)=\frac{i^2 - 3i+2}{2}, ~i \in N^0 _2.$

Hence,
$$
f(x)=2x-3f_{01} (x)=2x-3\left(x-\frac{1}{2}+f_{12} (x)\right)=\frac{3}{2}-x-3f_{12} (x).
$$

Since the following  relationship
$$
f_{01} (x)+f_{12} (x)+f_{02} (x)=\Delta^3 _{111...}=\frac{1}{2}
$$
holds for any  $x \in [0,1]$,  summing \eqref{ff2} and  \eqref{ff3}, we obtain
$$
2f(x)=x+\frac{3}{2}-3(f_{01} (x)+f_{12} (x))=x+\frac{3}{2}-3(\frac{1}{2}-f_{02} (x))=x+3f_{02} (x). 
$$
\end{proof}

\begin{lemma}
\label{lmf2}
The functions $f, f_{01},f_{02}, f_{12}$ have the following properties:
\begin{enumerate}
\item
$$
[0,1]\stackrel{f}{\rightarrow} \left([0,1] \setminus \{\Delta^3 _{\alpha_1\alpha_2...\alpha_n111...}\}\right) \cup\left\{\frac{1}{2}\right\};
$$
\item the point $x_0=0$ is a unique invariant point of the function  $f$;

\item the function $f$ is not bijective on a certain countable subset of $[0,1]$.
\item the following relationships hold for all $x \in [0,1]$:
\begin{equation}
\label{ff5}
f(x)-f(1-x)=f_{01} (x)-f_{12} (x),
\end{equation}
\begin{equation}
\label{ff6}
f(x)+f(1-x)=\frac{1}{2}+3f_{02} (x),
\end{equation}
\begin{equation}
\label{ff7}
f_{01} (x)+f_{02} (x)+f_{12} (x)=\frac{1}{2},
\end{equation}
\begin{equation}
\label{ff8}
2f_{01} (x)+f_{02} (x)=x,
\end{equation}
\begin{equation}
 \label{ff9}
f_{01} (x)- f_{12} (x)=x- \frac{1}{2};
\end{equation}

\item the function $f$ is not monotonic on the domain, in particular the function $f$ is a decreasing function on the set
$$
\{x: x_1<x_2 \Rightarrow (x_1=\Delta^3 _{c_1...c_{n_0}1\alpha_{n_0+2}\alpha_{n_0+3}...} \wedge  x_2=\Delta^3 _{c_1...c_{n_0}2\beta_{n_0+2}\beta_{n_0+3}...})\}, 
$$
 where $n_0 \in \mathbb Z^+$, $c_1, c_2,\dots, c_{n_0}$ is an ordered set of the ternary digits, $\alpha_{n_0+i} \in N^0 _2$,  $\beta_{n_0+i} \in N^0 _2$, $i \in \mathbb N$, 
and the function $f$ is an increasing function on the set
$$
\{x: x_1<x_2 \Rightarrow (x_1=\Delta^3 _{c_1...c_{n_0}0\alpha_{n_0+2}\alpha_{n_0+3}...} \wedge  x_2=\Delta^3 _{c_1...c_{n_0}r\beta_{n_0+2}\beta_{n_0+3}...})\}, 
$$
where $r \in \{1,2\}$.
\end{enumerate}
\end{lemma}
\begin{proof} \emph{The first and the second properties} of the function $f$ follows from \eqref{ff1}. 

Let us prove that \emph{ the third property} is true. Let $x_1=\Delta^3 _{\alpha_1\alpha_2...\alpha_n...}$ and $x_2=\Delta^3 _{\beta_1\beta_2...\beta_n...}$ be numbers such that $x_1 \ne x_2$. Find the following set 
$$
G=\{x: f(x_1)=f(x_2), x_1 \ne x_2\}.
$$
\begin{itemize}

\item Let $f(x_1)=f(x_2)=y_{1,2}$ be a ternary-irrational number. Then   
$$
\varphi (\alpha_{n})=\varphi (\beta_{n})~~\forall n \in \mathbb N
$$
and there exists  $n_{0}$ such that $\alpha_{n_0} \ne \beta_{n_0}$.
It follows from the last-mentioned inequality and  \eqref{ff1}  that $\varphi(\alpha_{n_0}) \ne \varphi(\beta_{n_0})$. 
This contradicts the equality $\varphi (\alpha_{n})=\varphi (\beta_{n})$.
Hence the set $G$ does not contain the set of ternary-irrational numbers. 

\item Let $f(x_1)=f(x_2)=y_{1,2}$ be a ternary-rational number. Then there exists  $n_0 \in \mathbb Z^+$ such that 
$$
y_{1,2}=\Delta^3 _{\varphi(\alpha_1)\varphi(\alpha_2)...\varphi(\alpha_{n_0})\varphi(\alpha_{n_0+1})000...}=\Delta^3 _{\varphi(\beta_1)\varphi(\beta_2)...\varphi(\beta_{n_0})(\varphi(\alpha_{n_0+1})-1)222...}
$$
or
$$
y_{1,2}=\Delta^3 _{\varphi(\beta_1)\varphi(\beta_2)...\varphi(\beta_{n_0})\varphi(\beta_{n_0+1})000...}=\Delta^3 _{\varphi(\alpha_1)\varphi(\alpha_2)...\varphi(\alpha_{n_0})(\varphi(\beta_{n_0+1})-1)222...}.
$$
That is
$$
\left[
\begin{aligned}
\left\{
\begin{aligned}
\varphi(\alpha_{n_0+2})=\varphi(\alpha_{n_0+3})=\dots &=0\\
\varphi(\beta_{n_0+2})=\varphi(\beta_{n_0+3})=\dots &=2\\
\varphi(\beta_{n_0+1})& = \varphi(\alpha_{n_0+1})-1\\
\end{aligned}
\right.\\
\left\{
\begin{aligned}
\varphi(\beta_{n_0+2})=\varphi(\beta_{n_0+3})=\dots &=0\\
\varphi(\alpha_{n_0+2})=\varphi(\alpha_{n_0+3})=\dots &=2\\
\varphi(\alpha_{n_0+1})& =\varphi(\beta_{n_0+1}) -1.\\
\end{aligned}
\right.
\end{aligned}
\right.
$$
Here $\varphi(\alpha_1)=\varphi(\beta_1), \varphi(\alpha_2)=\varphi(\beta_2), \dots ,\varphi(\alpha_{n_0})=\varphi(\beta_{n_0})$.

Hence, 
$$
\left[
\begin{aligned}
\left\{
\begin{aligned}
\alpha_{n_0+2}=\alpha_{n_0+3}=\dots &=0\\
\beta_{n_0+2}=\beta_{n_0+3}=\dots & = 1\\
\left[
\begin{aligned}
\left\{
\begin{aligned}
\alpha_{n_0+1}&=2\\
\beta_{n_0+1} & = 0\\
\end{aligned}
\right.\\
\left\{
\begin{aligned}
\alpha_{n_0+1}&=1\\
\beta_{n_0+1} & = 2\\
\end{aligned}
\right.\\
\end{aligned}
\right.\\
\end{aligned}
\right.\\
\left\{
\begin{aligned}
\alpha_{n_0+2}=\alpha_{n_0+3}=\dots &=1\\
\beta_{n_0+2}=\beta_{n_0+3}=\dots & = 0\\
\left[
\begin{aligned}
\left\{
\begin{aligned}
\alpha_{n_0+1}&=0\\
\beta_{n_0+1} & = 2\\
\end{aligned}
\right.\\
\left\{
\begin{aligned}
\alpha_{n_0+1}&=2\\
\beta_{n_0+1} & = 1.\\
\end{aligned}
\right.\\
\end{aligned}
\right.\\
\end{aligned}
\right.\\
\end{aligned}
\right.
$$
\end{itemize}

So $f(x_1)=f(x_2)$ for $x_1 \ne x_2$ on the following sets: 
\begin{itemize}
\item
$G _1=\left\{x: x_1=\Delta^3 _{c_1c_2...c_{n_0}2000...} \wedge x_2=\Delta^3 _{c_1c_2...c_{n_0}0111...}\right\},$
\item
$G _2=\left\{x: x_1=\Delta^3 _{c_1c_2...c_{n_0}1000...} \wedge x_2=\Delta^3 _{c_1c_2...c_{n_0}2111...}\right\},$
\item
$G _3=\left\{x: x_1=\Delta^3 _{c_1c_2...c_{n_0}0111...} \wedge x_2=\Delta^3 _{c_1c_2...c_{n_0}2000...}\right\},$
\item
$G _4=\left\{x: x_1=\Delta^3 _{c_1c_2...c_{n_0}2111...} \wedge x_2=\Delta^3 _{c_1c_2...c_{n_0}1000...}\right\},$ 
\end{itemize}
where $c_1, c_2,\dots , c_n$  are the fixed ternary digits, $n_0 \in \mathbb Z^+$. Since the set 
$$
G=G_1 \cup G_2 \cup G_3 \cup G_4
$$
 is a subset of the set of rational numbers, we obtain that $G$ is countable.

Let us prove that \emph{the fourth property} is true. Clearly,   $\varphi(i)=\frac{-3i^2+7i}{2}$, $i \in \{0,1,2\}$, and  $1=\Delta^3 _{222...}$. Consider the difference 
$$
\varphi(i)-\varphi(2-i)=\frac{-3i^2+7i}{2}-\frac{-3(2-i)^2+7(2-i)}{2}=
$$
$$
=\frac{-3i^2+7i+3(4-4i+i^2)-7(2-i)}{2}=\frac{2i-2}{2}=i-1.
$$
Similarly,  from the definitions of   $f_{01}$ and $f_{12}$ we have the following:
$$
\varphi_{01} (i) - \varphi_{12} (i)=\frac{i^2-i}{2}-\frac{i^2-3i+2}{2}=\frac{2i-2}{2}=i-1.
$$
So  relationship  (\ref{ff5}) holds.

It is easy to prove that   (\ref{ff6}) holds. Obviously,
$$
\varphi(i)+\varphi(2-i)=\frac{-3i^2+7i}{2}+\frac{-3(2-i)^2+7(2-i)}{2}=
$$
$$
=\frac{-3i^2+7i-3(4-4i+i^2)+14-7i}{2}=\frac{-6i^2+12i+2}{2}=
$$
$$
=-3i^2+6i+1=1-3(-i^2+2i)=\frac{1}{2}+3\varphi_{02} (i).
$$

Relationships \eqref{ff7} and \eqref{ff8} follow from the definitions of $f_{01}$, $f_{02}$, $f_{12}$.

Let us show that relationship \eqref{ff9} holds. Substracting  \eqref{ff3} from \eqref{ff2} and multiplying the obtained difference by $\frac{1}{3}$, we get the equality that is equivalent to  \eqref{ff9}.

Let  $x_1=\Delta^3 _{\alpha_1\alpha_2...\alpha_n...}$ and $x_2=\Delta^3 _{\beta_1\beta_2...\beta_n...}$. A certain function $f$ is called \emph{a decreasing function} whenever the inequality $f(x_1)>f(x_2)$ holds for any  $x_1 < x_2$ from the domain of $f$. In our case, there exists $n_0 \in \mathbb Z^+$ such that $\varphi(\alpha_{n_0+1})>\varphi(\beta_{n_0+1})$ for 
$\alpha_{n_0+1}<\beta_{n_0+1}$ and  $\alpha_1=\beta_1,\dots ,\alpha_{n_0}=\beta_{n_0}$. Hence from the definition of $f$ it follows that $\alpha_{n_0+1}=1 $ and $\beta_{n_0+1}=2$. Therefore the function $f$ is a decreasing function on the following set:
$$
\left\{x: x_1<x_2 \Rightarrow (x_1=\Delta^3 _{c_1...c_{n_0}1\alpha_{n_0+2}\alpha_{n_0+3}...} \wedge  x_2=\Delta^3 _{c_1...c_{n_0}2\beta_{n_0+2}\beta_{n_0+3}...})\right\}, 
$$
 where $c_1, c_2,\dots ,c_{n_0}$ are fixed ternary digits, $n_0$ is a fixed positive integer. 

A certain function $f$ is called \emph{an increasing function} whenever $f(x_1)<f(x_2)$ holds for any $x_1 <x_2$ from the domain of $f$. In this case, there exists $n_0 \in \mathbb Z^+$ such  that $\varphi(\alpha_{n_0+1})<\varphi(\beta_{n_0+1})$ for $\alpha_{n_0+1}<\beta_{n_0+1}$ and $\alpha_1=\beta_1,\dots ,\alpha_{n_0}=\beta_{n_0}$. Similarly, the function $f$ is an increasing function on the following set:
$$
\left\{x: x_1<x_2 \Rightarrow (x_1=\Delta^3 _{c_1...c_{n_0}0\alpha_{n_0+2}\alpha_{n_0+3}...} \wedge  x_2=\Delta^3 _{c_1...c_{n_0}r\beta_{n_0+2}\beta_{n_0+3}...})\right\}, 
$$
where $r \in \{1,2\}$. 
\end{proof}
\begin{theorem}
The function  $f$ satisfies the following functional equation:
\begin{equation*}
f(x)-f(1-x)=x-\frac{1}{2}.
\end{equation*}
\end{theorem}
\begin{proof} The statement follows from  \eqref{ff4}, \eqref{ff5}, \eqref{ff8},  and \eqref{ff9}. 
\end{proof}

\section{Differential properties}

\begin{theorem}
The function $f$ is continuous at ternary-irrational points, and ternary-rational points are points of discontinuity of the function.
Furthermore, a ternary-rational point $x_0=\Delta^3 _{\alpha_1\alpha_2...\alpha_n000...} $ is a  point of discontinuity $\frac{1}{2\cdot 3^{n-1}}$ whenever $\alpha_n=1$, and is a point of discontinuity $\left(-\frac{1}{2\cdot3^{n-1}}\right)$ whenever $\alpha_n=2$.
\end{theorem}
\begin{proof}
Let $x_0$ be a ternary-irrational number. Let us show that
$$
\lim_{x \to x_0} {|f(x)-f(x_0)|}=0.
$$

For an arbitrary $x=\Delta^3 _{\alpha_1\alpha_2...\alpha_n...} \in [0,1]$ there exists $n_0=n_0(x)$ such that 
$$
\left\{
\begin{aligned}
\alpha_m (x)&=\alpha_m (x_0), ~m=\overline{1,n_0-1}\\
 \alpha_{n_0} (x)& \ne \alpha_{n_0} (x_0),\\
\end{aligned}
\right.
$$
and the conditions $x \to x_0$  and $n_0 \to \infty$ are equivalent. Hence, 
$$
|f(x)-f(x_0)|=\left|\sum^{\infty} _{l=n_0} {\frac{\varphi(\alpha_{l}(x))-\varphi(\alpha_l (x_0))}{3^l}}\right| \le \sum^{\infty} _{l=n_0}{\frac{|\varphi(\alpha_{l}(x))-\varphi(\alpha_l (x_0))|}{3^l}}\le
$$
$$
\le \sum^{\infty} _{l=n_0} {\frac{2}{3^l}}=\frac{1}{3^{n_0-1}} \to 0~~~(n_0 \to \infty),
$$
Whence the function $f$  is continuous at  the point  $x_0$.

Let $x_0$ be a ternary-rational number. That is
$$
x_0=\Delta^3 _{\alpha_1\alpha_2...\alpha_{n-1}\alpha_n000...}=\Delta^3 _{\alpha_1\alpha_2...\alpha_{n-1}(\alpha_n-1)222...}, ~\alpha_n \ne 0.
$$
Then 
$$
\lim_{x \to x_0 -0} {f(x)}=\Delta^3 _{\varphi(\alpha_1)\varphi(\alpha_2)...\varphi(\alpha_{n-1})\varphi(\alpha_n-1)111...},
$$
$$
\lim_{x \to x_0 +0} {f(x)}=\Delta^3 _{\varphi(\alpha_1)\varphi(\alpha_2)...\varphi(\alpha_{n-1})\varphi(\alpha_n)000...},
$$
$$
\lim_{x \to x_0 +0} {f(x)}-\lim_{x \to x_0 -0} {f(x)}=\frac{\varphi(\alpha_n)-\varphi(\alpha_n-1)-1}{2\cdot 3^n}=\begin{cases}
\frac{1}{2\cdot 3^{n-1}}&\text{if $\alpha_n=1$}\\
-\frac{1}{2\cdot3^{n-1}}&\text{if $\alpha_n=2$.}
\end{cases}
$$
Whence $x_0$ is a point of discontinuity of  $f$. 
\end{proof}

\begin{theorem}
The function $f$ is nowhere differentiable. 
\end{theorem}
\begin{proof}
Since ternary-rational points are points of discontinuity of  $f$, we shall not consider these points. 

 Let $x_0$ be an arbitrary ternary-irrational number from  $[0;1]$. Since one of the digits is used infinitely many times in the ternary representation of $x_0$, we fix this digit $\alpha$ and fix one of the positions  $n_0$ such that $\alpha_{n_0}=\alpha$ in the ternary representation of $x_0$. That is $x_0=\Delta^3 _{\alpha_1\alpha_2...\alpha_{n_0-1}\alpha\alpha_{n_0+1}...}.$
Let  $(x_{n'})$ be a sequence of numbers   $x_{n'}$ such that
$$
x_{n'}=\Delta^3 _{\alpha_1\alpha_2...\alpha_{n_0-1}\beta_{n_0}\alpha_{n_0+1}...}.
$$
Then
$$
x_0-x_{n'}=\frac{\alpha-\beta_{n_0}}{3^{n_0}} ~\mbox{and} ~f(x_0)-f(x_{n'})=\frac{\varphi(\alpha)-\varphi(\beta_{n_0})}{3^{n_0}}.
$$
Hence,
$$
f' (x)=\lim_{\alpha \to \beta_{n_0}}{\frac{\varphi(\alpha)-\varphi(\beta_{n_0})}{\alpha-\beta_{n_0}}}.
$$
Since, for different values of $\alpha$ and $\beta_{n_0}\ne\alpha$, the derivative of $f$ at the point $x_0$ has different values, the function $f$ is nowhere differentiable.
\end{proof}

\section{Fractal properties of level sets}

Consider fractal properties of all level sets of the functions $f_{01}, f_{02}, f_{12}$. 

The following set  
$$
f^{-1} (y_0)=\{x: g(x)=y_0\},
$$
where $y_0$ is a fixed element of range of values $E(g)$ of the function $g$, is called \emph{a level set of  $g$}.

\begin{theorem}
\label{rivnif1}
The following statements are true:
\begin{itemize}
\item if there exists at least one digit $2$ in the ternary representation of $y_0$, then $f^{-1} _{ij} (y_0)=\varnothing$;

\item if $y_0=0$ or $y_0$  is a ternary-rational number from the set $C[3, \{0,1\}]$, then  
$$
\alpha_0(f^{-1} _{ij} (y_0))=\log_3 2;
$$

\item if $y_0$ is a ternary-irrational number from the set $C[3, \{0,1\}]$, then
$$
0 \le \alpha_0(f^{-1} _{ij} (y_0))\le \log_3 2,
$$
where  $\alpha_0(f^{-1} _{ij} (y_0))$ is the Hausdorff-Besicovitch  dimension of  $f^{-1} _{ij} (y_0)$.
\end{itemize}
\end{theorem}
\begin{proof}
From the definition of $f_{ij}$ it follows that $f^{-1} _{ij} (y_0)=\varnothing$ whenever there exists at least one digit  $2$ in the ternary representation of $y_0$. That is  range of values of  $f_{ij}$ is the following set
$$
E_1=\left\{x: x=\Delta^3 _{e_1e_2...e_n...}, e_n \in \{0,1\}\right\}=C[3, \{0,1\}].
$$

Similarly, if  $y_0=0$, then  the set of preimages is a set of Cantor type and a value of the Hausdorff-Besicovitch dimension  of this set equals    $\log_{3} 2$. The sets of preimages of $0$ under the functions $f_{01},f_{02},f_{12}$ are the following sets respectively: $E_1$,  $E_2=\{x: x=\Delta^3 _{u_1u_2...u_n...}, u_n \in\{0,2\}\},$ $E_3=\{x: x=\Delta^3 _{v_1v_2...v_n...}, v_n \in\{1,2\}\}.$

Let $y_0$  be a ternary-rational number from $E_1$, i.e., 
$$
y_0=\frac{1}{3^{l_1}}+\frac{1}{3^{l_2}}+\dots+\frac{1}{3^{l_{n_0}}}+\frac{0}{3^{{l_{n_0}+1}}}+\frac{0}{3^{{l_{n_0}+2}}}+\dots .
$$
Then the set of preimages of $y_0$ under the functions  $f_{01}$, $f_{02}$,  $f_{12}$ is the set of all numbers from $[0,1]$ such that:
 \begin{enumerate}
\item are numbers whose the ternary representation is periodic and this period contains digits from the sets $\{0,1\}$ (in the case of the function $f_{01}$), $\{0,2\}$ (in the case of $f_{02}$), $\{1,2\}$ (in the case of $f_{12}$) respectively;
\item are irrational numbers  such that 
their ternary representation contains the digits from $\{0,1\}, \{0,2\}, \{1,2\}$ respectively on positions $n_0+1, n_0+2,\dots$, where $n_0$ is a fixed number from $\mathbb N\cup \{0\}$.
\end{enumerate}

It is easy to prove that a value of the Hausdorff-Besicovitch dimension of such set equals $\log_3 {2}$. For example, consider the case of the function $y=f_{01} (x)$. Obviously,
$$
f^{-1} _{01} (y_0)=\left\{x: x=a_{l_{n_0}}+\frac{1}{3^{l_{n_0}}}\left(\frac{e_{l_{n_0}+1}}{3}+\frac{e_{l_{n_0}+2}}{3^2}+\dots+\frac{e_{l_{n_0}+m}}{3^m}+\dots\right)\right\}=
$$
$$
=\left\{x: x=a_{l_{n_0}} +\frac{1}{3^{l_{n_0}}} E_1\right\},~\mbox{where}~a_{l_n}=\Delta^3 _{e_1...e_{l_1-1}2e_{l_1+1}...e_{l_2-1}2e_{l_2+1}...e_{l_{n_0}-1}200...},
$$
$n_0$ is a fixed positive integer, $e_{l_n} \in\{0,1\}$, and $n \in \mathbb N$.

Since $n_0$ is a fixed number that depends only on $y_0$, we obtain
$$
\alpha_0 (f^{-1} _{01} (y_0))=\alpha_0 (E_1)=\log_{3}2,
$$
where $\alpha_0 (f^{-1} _{01} (y_0))$ is the Hausdorff-Besicivitch dimension of the set  $f^{-1} _{01} (y_0)$.

Consider the cases of the functions $f_{02}$ and  $f_{12}$.

 Let $y_0$ be a ternary-irrational number from  $E_1,$ i.e.,
$$
y_0=\frac{1}{3^{l_1}}+\frac{1}{3^{l_2}}+\dots+\frac{1}{3^{l_{n}}}+\dots .
$$
Then
$$
f^{-1} _{ij} (y_0)=\{x: x=\widehat{\Delta}^{l_1l_2...l_n...}  _{kk...k...}, l_n \in \mathbb N\},
$$
where $\widehat{\Delta}^{l_1l_2...l_n...}  _{kk...k...}$ is a number from  $[0,1]$ such that the digit $k$ is situated on the positions 
$l_1,l_2,\dots ,l_n,\dots ,$ and the digits from $\{i,j\}$ are situated on other positions in the ternary representation of this number.
Here $l_1,l_2,\dots ,l_n,\dots$, are fixed positive integers and depend  on $y_0$ such that $y_0=f_{ij} (x_0)$. That is the sequence $(l_n)$ is a given monotonic increasing sequence of positive integers.

Hence $0 \le \alpha_0(f^{-1} _{ij} (y_0)) \le \log_3 2$. This fact depends on the frequence of the digit  $1$  in the ternary representation of  $y_0$. Since
$$
\widehat{\Delta}^{l_1l_2...l_n...}  _{kk...k...}=\Delta^3 _{\underbrace{00...00}_{{l_1-1 }}k\underbrace{000...000}_{{l_2-l_1-1}}k... }+\widehat{\Delta}^{l_1l_2...l_n...}  _{00...0...},
$$
where $\Delta^3 _{\underbrace{00...00}_{{l_1-1}}k\underbrace{000...000}_{{l_2-l_1-1}}k... }$ is a fixed number that depens only on  $y_0$, and 
$\widehat{\Delta}^{l_1l_2...l_n...}  _{00...0...}$ is a subset of  $C[3,\{i, j\}]$ (for $k\ne 0$), and 
$$
\alpha_0(f^{-1} _{ij} (y_0)) \to 0 ~~~\mbox{as}~~~\frac{n}{l_n} \to 1,
$$
$$
\alpha_0(f^{-1} _{ij} (y_0)) \to \log_3 2 ~~~\mbox{as}~~~\frac{n}{l_n} \to 0,
$$
as $n \to \infty$.  The last two paragraphs are true  for  $k=0$ as well. 
\end{proof}

\section{Fractal properties of graph of considered function}

Suppose that 
$$
X=[0;1]\times[0;1]=\left\{(x,y): x=\sum^{\infty} _{m=1} {\frac{\alpha_m}{3^{m}}}, \alpha_{m} \in N^{0} _{2},
y=\sum^{\infty} _{m=1} {\frac{\beta_m}{3^{m}}}, \beta_{m} \in N^{0} _{2}\right\}.
$$
Then the set 
$$
\sqcap_{(\alpha_{1}\beta_{1})(\alpha_{2}\beta_{2})...(\alpha_{m}\beta_{m})}=\Delta^{3} _{\alpha_{1}\alpha_{2}...\alpha_{m}}\times\Delta^{3} _{\beta_{1}\beta_{2}...\beta_{m}}
$$
is a square with a side length of $3^{-m}$. This square is called \emph{a square of rank $m$ with the base $(\alpha_{1}\beta_{1})(\alpha_{2}\beta_{2})\ldots (\alpha_{m}\beta_{m})$}.

If  $E\subset X$, then the number
$$
\alpha^{K}(E)=\inf\{\alpha: \widehat{H}_{\alpha} (E)=0\}=\sup\{\alpha: \widehat{H}_{\alpha} (E)=\infty\},
$$
where
$$
\widehat{H}_{\alpha} (E)=\lim_{\varepsilon \to 0} \left[{\inf_{d\leq \varepsilon} {K(E,d)d^{\alpha}}}\right]
$$
and $K(E,d)$ is the minimum number of squares of diameter $d$ required to cover the set $E$, is called \emph{ the fractal cell entropy dimension of the set E.} It is easy to see that $\alpha^{K}(E)\ge \alpha_0(E)$.
\begin{theorem}
The Hausdorff-Besicovitch dimension of the graph of $f$  is equal to 1.
\end{theorem}
\begin{proof}
From the definition and properties of the function  $f$ it follows that the graph of the function belongs to $3$ squares from  $9$
first-rank squares:
$$
\sqcap_{(ij)}=\left[\frac{i}{3};\frac{i+1}{3}\right]\times\left[\frac{j}{3};\frac{j+1}{3}\right],~i \in N^0 _2, ~j \in N^0 _2,
$$
i.e., $\sqcap_{(00)},\sqcap_{(12)},\sqcap_{(21)}$.

The graph of  the function $f$ belongs to $9=3^2$ squares from $81=3^4$ second-rank squares:
$$
\sqcap_{(i_1j_2)(i_2j_2)}=\left[\frac{i_1}{3}+\frac{i_2}{3^2};\frac{i_1}{3}+\frac{i_2+1}{3^2}\right]\times\left[\frac{j_1}{3}+\frac{j_2}{3^2};\frac{j+1}{3}+\frac{j_2+1}{3^2}\right],~i \in N^0 _2, ~j \in N^0 _2,
$$
$i_1 \in N^0 _2, i_2 \in N^0 _2, j_1 \in N^0 _2, j_2 \in N^0 _2$, i.e.,
\begin{enumerate}
\item The part of the graph, which is in the square $\sqcap_{(00)}$, belongs to $3$ squares $\sqcap_{(00)(00)},\sqcap_{(00)(12)},  \sqcap_{(00)(21)}$;
\item  the part of the graph, which is in the square $\sqcap_{(12)}$, belongs to $3$ squares $\sqcap_{(12)(00)},\sqcap_{(12)(12)},  \sqcap_{(12)(21)}$;
\item the part of the graph, which is in the square $\sqcap_{(21)}$, belongs to $3$ squares $\sqcap_{(21)(00)},\sqcap_{(21)(12)},  \sqcap_{(21)(21)}$, etc.
\end{enumerate}

The graph $\Gamma_{f}$ of the function  $f$ belongs to $3^m$ squares of rank $m$ with side $3^{-m}$. Then
$$
\widehat{H}_{\alpha} (\Gamma_f)=\lim_{\overline{m \to \infty}} {3^m \left(\sqrt{3^{-2m}+3^{-2m}}\right)^{\alpha}}=\lim_{\overline{m \to \infty}} {3^m \left(2\cdot 3^{-2m}\right)^{\frac{\alpha}{2}}}=
$$
$$
=\lim_{\overline{m \to \infty}} {\left(3^{\frac{2m}{\alpha}-2m}\cdot 2\right)^{\frac{\alpha}{2}}}=\lim_{\overline{m \to \infty}} {\left(2^{\frac{\alpha}{2}}\cdot (3^{1-\alpha})^{m}\right)}.
$$

It is obvious that if $3^{(1-\alpha)m} \to 0$ for $\alpha >1$, and the graph of the function has self-similar properties, then  $\alpha^K (\Gamma_f)=\alpha_0(\Gamma_f)=1$. 
\end{proof}

\section{Lebesgue integral}

\begin{theorem}
The Lebesgue integral of the function $f$   is equal to $\frac{1}{2}$.
\end{theorem}
\begin{proof}
The conditions of existence of the Lebesgue integral are true for the function $f$. Since the function has self-similar properties, the Lebesgue integral $I$ of $f$ can be calculated by the following equality:
$$
I=3\frac{1}{3^2}+3I\frac{1}{3^2}=\frac{1}{3}I+\frac{1}{3},
$$
whence $I=\frac{1}{2}$.
\end{proof}

\section{Certain generalizations}

One can define $m=3!=6$ functions  determined on $[0,1]$ in terms of the ternary number system by the following way:
$$
\Delta^{3} _{\alpha_{1}\alpha_{2}...\alpha_{n}...}\stackrel{f_m}{\rightarrow} \Delta^{3} _{\varphi_m(\alpha_{1})\varphi_m(\alpha_{2})...\varphi_m(\alpha_{n})...},
$$
where the function $\varphi_m(\alpha_n)$ determined on an alphabet of the ternary number system and $f_m (x)$ is defined by the following table for each $m=\overline{1,6}$. 
\begin{center}
\begin{tabular}{|c|c|c|c|}
\hline
 &$ $ 0 &$ 1 $ & $2$\\
\hline
$\varphi_1 (\alpha_n) $ &$0$ & $1$ & $2$\\
\hline
$\varphi_2 (\alpha_n) $ &$0$ & $2$ & $1$\\
\hline
$\varphi_3 (\alpha_n) $ &$1$ & $0$ & $2$\\
\hline
$\varphi_4 (\alpha_n) $ &$1$ & $2$ & $0$\\
\hline
$\varphi_5 (\alpha_n) $ &$2$ & $0$ & $1$\\
\hline
$\varphi_6 (\alpha_n) $ &$2$ & $1$ & $0$\\
\hline
\end{tabular}
\end{center}
That is one can to model a class of functions whose values  are obtained from the ternary representation
of the argument by  a certain  change of ternary digits.

It is easy to see that the function $f_1 (x)$ is the function  $y=x$ and the function $f_6 (x)$ is the function $y=1-x$, i. e.,
$$
y=f_1(x)=f_1\left(\Delta^3 _{\alpha_1\alpha_2\ldots \alpha_n\ldots}\right)=\Delta^3 _{\alpha_1\alpha_2\ldots \alpha_n\ldots}=x,
$$
$$
y=f_6(x)=f_6\left(\Delta^3 _{\alpha_1\alpha_2\ldots \alpha_n\ldots}\right)=\Delta^3 _{[2-\alpha_1][2-\alpha_2]\ldots  [2-\alpha_n]\ldots}=1-x.
$$

\begin{lemma} 
Any function  $f_m$ can be represented by the functions $f_{ij}$ in the following form
Довільна функція виражається за допомогою функцій  наступним чином
$$
f_m=a^{(ij)} _{m}x+b^{(ij)} _{m}+c^{(ij)} _{m}f_{ij} (x), ~\mbox{where}~a^{(ij)} _{m}, b^{(ij)} _{m}, c^{(ij)} _{m} \in \mathbb Q.
$$
\end{lemma}

One can to formulate the following corollary.

\begin{theorem}
The function $f_m $ such that $f_m(x)\ne x$ and $f_m(x)\ne1-x$ is:
\begin{itemize}
\item continuous almost everywhere;
\item nowhere differentiable;
\item a function whose the Hausdorff-Besicovitch dimension of the grapf is equal to  $1$;
\item  a function whose the Lebesgue integral is equal to $\frac{1}{2}$.
\end{itemize}
\end{theorem}

\end{document}